\documentclass[11pt]{amsart}
\usepackage{amsmath}
\usepackage{amssymb, comment}
\usepackage{color}
\usepackage[dvipsnames]{xcolor}
\usepackage{soul}
\usepackage[margin=3cm]{geometry}

\usepackage[hidelinks]{hyperref}

\usepackage{cleveref}

\usepackage{quiver}
\usepackage{esint}
\usepackage{nicematrix}

\title{A Generalized Frankel Conjecture Via The Yang-Mills Flow}
\author[Jiangtao Li]{Jiangtao Li}
\date{\today}

\address{Department of Mathematics, University of California, San Diego, 9500 Gilmann Drive, La Jolla, CA 92092}
\email{jil320@ucsd.edu}

\theoremstyle{plain}
\newtheorem{thm}{Theorem}[section]
\newtheorem{prop}[thm]{Proposition}
\newtheorem{defn}[thm]{Definition}
\newtheorem{lem}[thm]{Lemma}
\newtheorem{cor}[thm]{Corollary}
\newtheorem{conj}[thm]{Conjecture}

\theoremstyle{definition}

\newtheorem{rk}[thm]{Remark}

\numberwithin{equation}{section}

\DeclareMathOperator{\R}{\mathbb{R}}
\DeclareMathOperator{\C}{\mathbb{C}}

\DeclareMathOperator{\Ricci}{Ric}

\DeclareMathOperator{\SO}{SO}
\DeclareMathOperator{\tr}{Tr}
\DeclareMathOperator{\Imagine}{Im}

\DeclareMathOperator{\End}{End}
\DeclareMathOperator{\Id}{Id}
\DeclareMathOperator{\YM}{YM}
\DeclareMathOperator{\vol}{vol}

\newcommand{\Li}[1]{\mathfrak{#1}}

\begin{document}

\maketitle

\begin{abstract}
    In this note, we introduce a new curvature condition called the $2-$positive bisectional curvature on compact K\"{a}hler manifolds. We then deduce a characterization theorem for manifolds with $2-$positive bisectional curvature, which can be regarded as a variant of the classical Frankel conjecture (cf.\cite{Fra61,SY80}) and its generalizations (cf.\cite{Siu80,Mok88}).
\end{abstract}

\tableofcontents

\section{Introduction}\label{Sec:Introduction}

One of the landmark results in complex differential geometry is the resolution of the Frankel's conjecture (cf.\cite{Fra61}) by Siu and Yau:
\begin{thm}[\cite{SY80}]\label{thm:FrankelConjecture}
    A compact K\"{a}hler manifold with positive bisectional curvature is biholomorphic to the complex projective space.
\end{thm}

A purely algebro-geometric generalization of the conjecture was proposed by Hartshorne (cf.\cite{Har70}) and proved by Mori (cf.\cite{Mor79}). This is one of the first few results in the modern study of Fano manifolds and Mori's original technique has tremendous influence on the later development of the theory on Fano manifolds.

On the differential geometric side, \Cref{thm:FrankelConjecture} has also been generalized in multiple different ways. Under the non-compact setting, it was generalized by the famous Yau's uniformization conjecture:
\begin{conj}[\cite{Yau88,SY94}]
    A non-compact K\"{a}hler manifold with positive bisectional curvature is biholomorphic to the complex Euclidean space.
\end{conj}
\noindent The conjecture has received a lot of attention during the last decades and there has been much progress towards it by many mathematicians (see \cite{Ni98,NT03,Ni04,CFYZ06,Liu16}). The best result so far was obtained by G. Liu, who confirmed the conjecture under the additional assumption of maximum volume growth. We refer the readers to the survey \cite{Liu20} for more details on this problem.

There are also various generalizations in the compact case. For instance, Siu (cf.\cite{Siu80}) gave a curvature characterization of complex projective spaces and smooth complex hyperquadrics using similar methods as in \cite{SY80}. Later, Mok (cf.\cite{Mok88}) studied compact K\"{a}hler manifolds with non-negative bisectional curvature and showed that they are exactly all the compact Hermitian symmetric spaces. In addition, it was shown by Gu-Zhang (cf.\cite{GZ10}) and Feng-Liu-Wan (cf.\cite{FLW17}) that the Frankel's conjecture still holds if the assumption of positive bisectional curvature is weakened to positive orthogonal bisectional curvature.

In this note we give another generalization of the classical Frankel's conjecture similar to that of Siu (cf.\cite{Siu80}). We shall introduce a new curvature condition for compact K\"{a}hler manifolds of dimension at least 2, which we call $2-$positive bisectional curvature (see Definition \ref{defn:2PositiveBisectionalCurvature}). We remark that \cite{Siu80} also defined $m-$positivity of bisectional curvature but it is a different curvature condition from ours. One of differences is that the $m-$positivity in \cite{Siu80} requires non-negativity of the bisectional curvature while ours in this paper doesn't impose this restriction. The main theorem of this paper gives a characterization of compact K\"{a}hler manifolds with $2-$positive bisectional curvature.

\begin{thm}\label{thm:MainTheorem}
    Let $n\ge 2$. An $n$ dimensional compact K\"{a}hler manifold $M^n$ admits a K\"{a}hler metric with $2-$positive bisectional curvature if and only if
    \begin{enumerate}
        \item $M^2$ is biholomorphic to a del Pezzo surface if $n=2$;
        \item $M^n$ is biholomorphic to either $\C P^n$ or the smooth complex hyperquadric $Q^n$ if $n\ge 3$.
    \end{enumerate}
\end{thm}

As we shall see in \Cref{Sec:2PositiveBisectionalCurvature}, (1) of \Cref{thm:MainTheorem} follows directly from the definition of $2-$positivity and Yau's theorem (cf.\cite{Yau78}). The main part of the note is devoted to the proof of (2). We briefly mention the idea and key ingredients in the proof here. Firstly, by a classical result in algebraic geometry on the pseudoindex of Fano manifolds (cf.\cite{Miy04,DH17}), the statement is equivalent to an estimate on the minimal anti-canonical degree of rational curves on $M^n$. To establish the desired estimate, we consider the Yang-Mills flow on the Riemann sphere $\C P^1$. The long time existence and convergence are well known due to existing theory on the Yang-Mills flow on Riemann surfaces (cf.\cite{Don85,Das92,Rd92}). The key will be Corollary \ref{cor:MaximumPrincipleYangMills}, a maximum principle for the Yang-Mills flow, which follows from Hamilton's tensor maximum principle \Cref{thm:HamiltonMaximumPrinciple}. The Hamilton's maximum principle is well known and used widely in the applications of the Ricci flow. However, it seems never appearing in the Yang-Mills setting.

The paper is organized as follows: In \Cref{Sec:2PositiveBisectionalCurvature}, we define the curvature condition of $2-$positive bisectional curvature and construct such metrics on del Pezzo surfaces and smooth complex hyperquadrics. In \Cref{Sec:YangMills}, we give a review of the fundamental results on the Yang-Mills flow on Riemann surfaces, which will be used in the proof of the main theorem. Afterwards, we establish the key maximum principle Corollary \ref{cor:MaximumPrincipleYangMills}. In \Cref{Sec:Proof}, we complete the proof of the main theorem. Finally, we discuss some related questions and also relevant works by others in the last section.

\section{$2$-positive bisectional curvature}\label{Sec:2PositiveBisectionalCurvature}

Let $(M,g)$ be an $n$ dimensional compact K\"{a}her manifold. Let $TM$ be the holomorphic tangent bundle of $M$. At any point $x\in M$, we fix any holomorphic tangent vector $U\in T_xM$. Then $R(U,\Bar{U})$ is a Hermitian operator on $T_xM$ and thus diagonalizable with real eigenvalues 
\[\lambda_1(x,U)\leq\cdots\leq\lambda_n(x,U).\]
We introduce a new notion of positivity for bisectional curvatures which we call $2-$positive bisectional curvature:

\begin{defn}[$2$-positive bisectional curvature]\label{defn:2PositiveBisectionalCurvature}
    $(M,g)$ has $2$-positive bisectional curvature if $\lambda_1(x,U)+\lambda_2(x,U)>0$ for any $x\in M$ and any $U\in T_xM$. 
\end{defn}

We give some examples of compact K\"{a}hler manifolds with $2$-positive bisectional curvature. Observe that when $n=2$, $2$-positivity of bisectional curvature is equivalent to the positivity of Ricci curvature. Therefore, by Yau's resolution of the Calabi conjecture (cf.\cite{Yau78}), all Fano surfaces, namely, del Pezzo surfaces, admit K\"{a}hler metrics with $2$-positive bisectional curvature.

When $n>2$, it is well known that the Fubini-Study metric on the complex projective space has positive bisectional curvature and in particular has $2$-positive bisectional curvature. Next we give another example of compact K\"{a}hler manifolds with positive $2$-positive bisectional curvature. Indeed, as shown in the main theorem, these are the only manifolds with $2-$positive bisectional curvature in higher dimensions.

Recall that a \emph{complex hyperquadric} is the following nonsingular hypersurface in $\C P^{n+1}$:
\[Q^n=\left\{[X_0:\cdots:X_{n+1}]\in\C P^{n+1}:X_0^2+\cdots+X_{n+1}^2=0\right\}\subseteq \C P^{n+1}.\]
Note that $Q^n$ can be regarded compact Hermitian symmetric space. In fact, we have 
\begin{equation}
    Q^n\cong \frac{\SO(n+2,\R)}{\SO(n,\R)\times\SO(2,\R)}.
\end{equation}

We will calculate the bisectional curvature of $Q^n$ equipped with the Hermitian symmetric metric and show that it also has $2$-positive bisectional curvature. This was carried out in \cite{Siu80} using the duality between compact and non-compact Hermitian manifolds. Here we give a more direct computation. The basic geometric theory of Riemannian symmetric spaces is needed. For this, we refer readers to the standard texts, e.g. \cite[Chapter 8]{Wol11}, \cite{Hel78}.

For the Hermitian symmetric space $Q^n$, the Cartan decomposition is 
\begin{equation}
    \Li{so}(n+2,\R)=\Li{so}(n,\R)\oplus\Li{so}(2,\R)\oplus\Li{p}.
\end{equation}
where
\begin{equation}
    \Li{p}=\left\{\begin{pmatrix}
        O&-X^T\\
        X&O
    \end{pmatrix}:X\in M_{2\times n}(\R)\right\}.
\end{equation}
The tangent space of $Q^n$ at the point $\SO(n,\R)\times\SO(2,\R)$ can be identified with the subspace $\Li{p}$, which can be further identified with $M_{2\times n}(\R)$. Take $M_{2\times n}(\R)$ as the tangent space. The complex structure $J$ is then given as follows:
\begin{align*}
    J\begin{pNiceMatrix}[last-row]
        0&\cdots&1&\cdots&0\\
        0&\cdots&0&\cdots&0\\
        &&i&&
    \end{pNiceMatrix}
    &=J\begin{pNiceMatrix}[last-row]
        0&\cdots&0&\cdots&0\\
        0&\cdots&1&\cdots&0\\
        &&i&&
    \end{pNiceMatrix},\\
    J\begin{pNiceMatrix}[last-row]
        0&\cdots&0&\cdots&0\\
        0&\cdots&1&\cdots&0\\
        &&i&&
    \end{pNiceMatrix}
    &=J\begin{pNiceMatrix}[last-row]
        0&\cdots&-1&\cdots&0\\
        0&\cdots&0&\cdots&0\\
        &&i&&
    \end{pNiceMatrix},\\
    i&= 1,\cdots, n.
\end{align*}
The complexified tangent space is $M_{2\times n}(\C)$ and it decomposes into holomorphic and conjugate holomorphic parts as below:
\begin{equation}
    M_{2\times n}(\C)=\left\{\begin{pmatrix}
        a_1&\cdots&a_n\\
        -\sqrt{-1}a_1&\cdots&-\sqrt{-1}a_n
    \end{pmatrix}:a_i\in\C\right\}\oplus\left\{\begin{pmatrix}
        a_1&\cdots&a_n\\
        \sqrt{-1}a_1&\cdots&\sqrt{-1}a_n
    \end{pmatrix}:a_i\in\C\right\}.
\end{equation}

The Hermitian symmetric metric on the tangent space, which is identified with $\Li{p}$, is simply the restriction $-\lambda B$ onto $\Li{p}$, where $B$ is the Killing form of $\Li{so}(n+2,\R)$ and $\lambda$ is any positive constant. For simplicity, we let $\lambda=\frac{1}{2}$. Then the metric $g$ is
\begin{equation*}
    g\left(\begin{pmatrix}
        O&-X^T\\
        X&O
    \end{pmatrix},
    \begin{pmatrix}
        O&-Y^T\\
        Y&O
    \end{pmatrix}\right)=-\frac{1}{2}\tr\left(\begin{pmatrix}
        O&-X^T\\
        X&O
    \end{pmatrix}
    \begin{pmatrix}
        O&-Y^T\\
        Y&O
    \end{pmatrix}\right)=\tr XY^T.
\end{equation*}
If we regard $M_{2\times n}(\R)$ as the tangent space, we get the following expression of the metric 
\[g(X,Y)=\tr XY^T,\quad \forall X,Y\in M_{2\times n}(\R).\]

Let $R$ be the Riemann curvature tensor of $Q^n$. Under the identification of the tangent space with $\Li{p}$, the curvature formula for Riemannian symmetric spaces (see \cite[\S 8.4]{Wol11}) implies that
\begin{align*}
    &R\left(\begin{pmatrix}
        O&-X^T\\
        X&O
    \end{pmatrix},\begin{pmatrix}
        O&-Y^T\\
        Y&O
    \end{pmatrix}\right)\begin{pmatrix}
        O&-Z^T\\
        Z&O
    \end{pmatrix}\\
    =&-\left[\left[\begin{pmatrix}
        O&-X^T\\
        X&O
    \end{pmatrix},\begin{pmatrix}
        O&-Y^T\\
        Y&O
    \end{pmatrix}\right],\begin{pmatrix}
        O&-Z^T\\
        Z&O
    \end{pmatrix}\right]\\
    =&\begin{pmatrix}
        O&-X^TYZ^T-Z^TYX^T+Y^TXZ^T+Z^TXY^T\\
        ZY^TX+XY^TZ-ZX^TY-YX^TZ&O
    \end{pmatrix}.
\end{align*}
Therefore, if we identify the tangent space with $M_{2\times n}(\R)$, then the curvature tensor is 
\begin{equation}\label{eq:CurvatureFormula}
    R(X,Y)Z=ZY^TX+XY^TZ-ZX^TY-YX^TZ, \quad \forall X,Y,Z\in M_{2\times n}(\R).
\end{equation}

Given the above discussion, we can get the following expression for the bisectional curvature:
\begin{lem}\label{lem:BisectionalCurvature}
    Let 
    \[U=\begin{pmatrix}
        a_1&\cdots&a_n\\
        -\sqrt{-1}a_1&\cdots&-\sqrt{-1}a_n
    \end{pmatrix},\quad V=\begin{pmatrix}
        b_1&\cdots&b_n\\
        -\sqrt{-1}b_1&\cdots&-\sqrt{-1}b_n
    \end{pmatrix}\]
    be holomorphic tangent vectors. Then 
    \begin{equation}\label{eq:BisectionalCurvature}
        R(U,\Bar{U},V,\Bar{V})=4\left(\sum_{i=1}^n|a_i|^2\right)\left(\sum_{i=1}^n|b_i|^2\right)-16\sum_{i<j}\Imagine(\Bar{a_i}a_j)\Imagine(b_i\Bar{b_j}).
    \end{equation}
\end{lem}
\begin{proof}
    By the curvature formula \Cref{eq:CurvatureFormula} and the metric expression, we get
    \begin{align*}
        R(U,\Bar{U},V,\Bar{V})=\tr(V\Bar{U}^TU\Bar{V}^T+U\Bar{U}^TV\Bar{V}^T-VU^T\Bar{U}\Bar{V}^T-\Bar{U}U^TV\Bar{V}^T).
    \end{align*}
    \Cref{eq:BisectionalCurvature} then follows from direct calculation.
\end{proof}

As a corollary, we can deduce the following $2-$positivity property of bisectional curvature for $Q^n$:
\begin{cor}
    The Hermitian symmetric space $Q^n$ has nonnegative bisectional curvature. More precisely, for any holomorphic tangent vector $U\in T^{1,0}_xQ^n$, $R(U,\Bar{U}):T_x^{1,0}Q^n\to T_x^{1,0}Q^n$ has nonnegative eigenvalues and at most one of them is zero. In particular, $Q^n$ has $2$-positive bisectional curvature.
\end{cor}
\begin{proof}
    Observe that the curvature is invariant under the adjoint action of the isotropy group $\SO(n,\R)\times\SO(2,\R)$. Therefore, we have 
    \begin{equation}\label{eq:InvarianceOfCurvature}
    R(BUA,B\Bar{U}A,BVA,B\Bar{V}A)=R(U,\Bar{U},V,\Bar{V}), \quad \forall A\in\SO(n,\R), B\in\SO(2,\R).
    \end{equation}
    Suppose that $U,V$ are nonzero holomorphic tangent vectors as in Definition \ref{lem:BisectionalCurvature}. By \ref{eq:InvarianceOfCurvature}, we can assume, without loss of generality, that $a_i$'s are all zero except for $a_1$ and $a_2$. Next we split into two cases:
    \begin{enumerate}
        \item If one of $a_1$ and $a_2$ is zero, then it follows from \Cref{eq:BisectionalCurvature} that 
        \[R(U,\Bar{U},V,\Bar{V})=4\left(\sum_{i=1}^n|a_i|^2\right)\left(\sum_{i=1}^n|b_i|^2\right)> 0.\]
        \item If both $a_1$ and $a_2$ are nonzero, then 
        \begin{align*}
            &R(U,\Bar{U},V,\Bar{V})\\
            =&4\left(\sum_{i=1}^n|a_i|^2\right)\left(\sum_{i=1}^n|b_i|^2\right)-8\sum_{i\ne j}^n\Imagine(\Bar{a_i}a_j)\Imagine(b_i\Bar{b_j})\\
            =&4\left(\sum_{i=1}^n|a_i|^2\right)\left(\sum_{i=1}^n|b_i|^2\right)-16\Imagine(\Bar{a_1}a_2)\Imagine(b_1\Bar{b_2})\\
            \ge&4\left(\sum_{i=1}^n|a_i|^2|b_i|^2+\sum_{i\ne j}|a_i|^2|b_j|^2\right)-16|a_1||a_2||b_1||b_2|\\
            =&4\left(\sum_{i=1}^2|a_i|^2|b_i|^2-2|a_1||b_1||a_2||b_2|\right)+4\left(\sum_{j\ne 1}|a_1|^2|b_j|^2+\sum_{j\ne 2}|a_2|^2|b_j|^2-2|a_1||b_1||a_2||b_2|\right)\\
            =&4(|a_1||b_1|-|a_2||b_2|)^2+(|a_1||b_2|-|a_2||b_1|)^2+4\left(\sum_{j\ne 1,2}|b_j|^2\right)(|a_1|^2+|a_2|^2)\\
            \ge&0.
        \end{align*}
        The equality holds if and only if 
        \begin{equation*}
            U=
            \begin{pmatrix}
                a&\sqrt{-1}a&\cdots&0\\
                -\sqrt{-1}a&a&\cdots&0
            \end{pmatrix},
            \quad
            V=\begin{pmatrix}
                b&-\sqrt{-1}b&\cdots&0\\
                -\sqrt{-1}b&-b&\cdots&0
            \end{pmatrix}, \quad \forall a,b\in\C^\times.
        \end{equation*}
    \end{enumerate}
    It follows from the above argument that $R$ has nonnegative and $2$-positive bisectional curvature.
\end{proof}

\section{The Yang-Mills flow on Riemann surfaces}\label{Sec:YangMills}

In this section, we present some fundamental results on the Yang Mills flow on Riemann surfaces, which are well known in literature (cf.\cite{BL81,AB83,Don83,Rd92,Das92}). We also establish a new maximum principle for the Yang-Mills flow in general as a corollary of Hamilton's maximum principle. The Hamilton's maximum principles is a standard and common tool in the field of the Ricci flow and there has been many successful applications in both general Ricci flow (cf.\cite{Ham86,BW08,BS09}) and K\"{a}hler-Ricci flow (cf.\cite{Che07,Ni07,NZ25}).

\subsection{Fundamentals of the Yang-Mills flow on Riemann surfaces}
Let $(M,g)$ be a compact K\"{a}hler manifold with K\"{a}hler form $\omega$ and $E$ be a rank $r$ complex vector bundle over $M$ with a Hermitian metric $h$. A \emph{unitary connection} on $(E,h)$ is a complex linear connection which is compatible with the Hermitian metric $h$. Let $D_A$ be the exterior derivative induced by a unitary connection $A$ and $D_A=D_A'+D_A''$ be its decomposition into holomorphic part $D_A'$ and conjugate holomorphic part $D_A''$. The connection is \emph{integrable} if $D_A''^2=0$. In this case, $D_A''$ induces a unique holomorphic structure on $E$ (cf.\cite[Proposition 1.4.17]{Kob87}).

We will only consider integrable unitary connections. For such a connection $A$ on $(E,h)$, the \emph{Yang-Mills energy} of $A$ is defined as the $L^2$ norm of the curvature form:
\begin{equation}
    \YM(A)=\int_M|F_A|^2\vol_g.
\end{equation}
The critical points of the Yang-Mills energy functional $\YM$ are called \emph{Yang-Mills connections}. A simple calculation of the first variation (cf.\cite[Theorem 2.21]{BL81}) shows that an integrable unitary connection $A$ is a Yang-Mills connection if and only if the following equation holds
\begin{equation}\label{eq:YangMillsEquation}
    D_A^*F_A=0.
\end{equation}
where $D_A^*$ is the formal adjoint to $D_A$ with respect to the $L^2$ inner product induced by $g$ and $h$.
\Cref{eq:YangMillsEquation} is called \emph{the Yang-Mills equation}.

The \emph{Yang-Mills flow} is the $L^2$ gradient flow of the Yang-Mills energy functional $\YM$. In other words, it is a time dependent connection $A(t)$ such that the following evolution equation holds:
\begin{equation}\label{eq:YangMillsFlow}
    \frac{\partial A}{\partial t}=-D_A^*F_A.
\end{equation}

The long time existence of the flow on compact K\"{a}hler manifolds was established by Donaldson (cf.\cite[Proposition 20]{Don85}):

\begin{thm}
    For any initial integrable and unitary connection $A$, the solution to the Yang-Mills flow \ref{eq:YangMillsFlow} exists for all $t\ge 0$.
\end{thm}

If we specialize to the case when the base K\"{a}hler manifold is one dimensional, namely, a Riemann surface, the following convergence result is known (cf.\cite{Rd92,Das92}):

\begin{thm}\label{thm:ConvergenceTheorem}
    Let $(M,g)$ be a Riemann surface and $(E,h)$ be a Hermitian complex vector bundle. If $A(t)$ is a solution to the Yang-Mills flow. Then $A(t)$ converges to a Yang-Mills connection $A_\infty$ smoothly as $t\to\infty$.
\end{thm}

\subsection{A maximum principle for Yang-Mills flow}

The curvature form $F_A$ is an skew-Hermitian endomorphism valued $(1,1)$ form. In other words, it is a section of the bundle $\Omega^2(\Li{u}(E))\subset\Omega^2(\End E)$. Let $\Lambda_\omega$ be the contraction with the K\"{a}hler form $\omega$. Then $\sqrt{-1}\Lambda_\omega F_A$ is a section of the vector bundle $\sqrt{-1}\Li{u}(E)$ consisting of Hermitian endomorphisms. We establish a tensor maximum principle for $\sqrt{-1}\Lambda_\omega F_A$ along the Yang-Mills flow in this subsection. 

Firstly, the evolution equation of the curvature form $F_A$ can be calculated from the Yang-Mills flow \ref{eq:YangMillsFlow}:
\begin{lem}
    The curvature form $F_A$ evolves by the following equation:
    \begin{equation}
        \frac{\partial F_A}{\partial t}=-\Delta_{D_A}F_A.
    \end{equation}
    where $\Delta_A=D_AD_A^*+D_A^*D_A$ is the Hodge Laplacian on $\End E$ valued forms.
\end{lem}
\begin{proof}
    We work locally. Choose a local unitary frame $\{\xi_\alpha\}$ of $E$ with respect to the fixed metric $h$. Then we may identify the connection form $A$ as a matrix valued 1 form and $F_A$ as a matrix valued 2 form. The Yang-Mills equation is then equivalent to 
    \[\frac{\partial A}{\partial t}=-d^*F_A\]
    where the RHS means that $-d^*$ acts on each entry of $F_A$. Since $F_A=dA+A\wedge A$, we deduce that
    \begin{align*}
        \frac{\partial F_A}{\partial t}& =d\left(\frac{\partial A}{\partial t}\right)+\frac{\partial A}{\partial t}\wedge A+A\wedge\frac{\partial A}{\partial t}\\
        & =-dd^*F_A-(d^*F_A)\wedge A-A\wedge(d^*F_A)\\
        & =-D_A(D_A^*F_A)=-\Delta_{D_A}F_A.
    \end{align*}
    Note that the last equality follows from the Bianchi identity $D_AF_A=0$.
\end{proof}

It follows from the K\"{a}hler identity $[\Delta_A,\Lambda_\omega]=\Delta_A\Lambda_\omega-\Lambda_\omega\Delta_A=0$ that the evolution equation of $\sqrt{-1}\Lambda_\omega F_A$ is:
\begin{equation}\label{eq:EvolutionContractedCurvature}
    \frac{\partial}{\partial t}\sqrt{-1}\Lambda_\omega F_A=-\Delta_A (\sqrt{-1}\Lambda_\omega F_A)=\Delta(\sqrt{-1}\Lambda_\omega F_A).
\end{equation}
where $\Delta=g^{i\Bar{j}}\nabla_i\nabla_{\Bar{j}}$ is the Laplace-Beltrami operator on $\End E$ and $\nabla$ is the connection induced by $A(t)$.

Recall the following general maximum principle for sections of vector bundles by R. Hamilton:
\begin{thm}[\text{\cite[Lemma 8.1]{Ham86}}]\label{thm:HamiltonMaximumPrinciple}
    Let $M$ be a closed Riemannian manifold with time dependent metric $g(t)$ and $E$ be a vector bundle over $M$ with a fixed metric $h$. Suppose that $\nabla(t)$ is a time dependent linear connection on $E$ which is compatible with $h$. Let $f$ be a section of $E$ satisfying the parabolic evolution equation
    \begin{equation}\label{eq:ParabolicEvolutionEquation}
        \frac{\partial f}{\partial t}=\Delta f+\phi(f)
    \end{equation}
    where $\Delta=g^{ij}\nabla_i\nabla_j$ is the Laplacian induced by $g(t)$ and $\nabla(t)$. For a subset $C$ of $E$ which is invariant under parallel translation, let $C_x=C\cap E_x$. If $C_x$ is closed, convex and preserved by the ODE 
    \begin{equation}
        \frac{df}{dt}=\phi(f),
    \end{equation}
    then $C$ is preserved by the flow \ref{eq:ParabolicEvolutionEquation}. Namely, if $f_x\in C_x, \forall x\in M$ at $t=0$, then it remains so for all $t\ge 0$. Moreover, if $f_x$ is in the interior of $C_x$ at some point $x\in M$ and $t=0$, then $f_x$ is in the interior of $C_x$ for any $x\in M$ all $t>0$.
\end{thm}

For any section $f\in\sqrt{-1}\Li{u}(E)$, its value $f_x$ at $x\in M$ is Hermitian and hence orthogonally diagonolizable with real eigenvalues $\lambda_1(f_x)\leq\cdots\leq\lambda_r(f_x)$. We say that 
\begin{enumerate}
    \item $f$ is $2$-nonnegative if $\lambda_1(f_x)+\lambda_2(f_x)\ge 0$ for all $x\in M$;
    \item $f$ is $2$-positive if $\lambda_{x,1}(f)+\lambda_{2,x}(f)>0$ for all $x\in M$;
    \item $f$ is $2$-quasi-positive if $\lambda_1(f_x)+\lambda_2(f_x)\ge 0$ for all $x\in M$ and the inequality is strict at one point of $M$;
    \item For $\varepsilon>0$, $f$ is $\varepsilon$ $2-$positive if $\lambda_1(f_x)+\lambda_2(f_x)\ge\varepsilon$ for all $x\in M$.
\end{enumerate}  

By applying the above general maximum principle \ref{thm:HamiltonMaximumPrinciple}, we obtain the following maximum principle for the Yang-Mills flow:
\begin{cor}\label{cor:MaximumPrincipleYangMills}
    Under the Yang-Mills flow \ref{eq:YangMillsFlow}, the following hold:
    \begin{enumerate}
        \item if $\sqrt{-1}\Lambda_\omega F_A$ is $\varepsilon$ $2-$positive at $t=0$, then it remains so for all $t\ge0$;
        \item If $\sqrt{-1}\Lambda_\omega F_A$ is quasi $2-$positive at $t=0$, then it is $2-$positive for all $t>0$.
    \end{enumerate} 
\end{cor}
\begin{proof}
    For (1), let $C_{\varepsilon}(E)$ be the subset of $\sqrt{-1}\Li{u}(E)$ consists of $\varepsilon$ $2-$positive Hermitian endomorphisms. We use the language of principal bundles (cf.\cite[Chapter II]{KN63}). Let $P$ be the unitary frame bundle of $(E,h)$, which is a principal $U(r)$ bundle over $M$. Let $\sqrt{-1}\Li{u}(r)$ be the vector space of all Hermitian matrices. The standard representation of $U(r)$ on $\C^r$ induces a representation $\rho$ of $U(r)$ on $\sqrt{-1}\Li{u}(r)$. Then there is the isomorphism
    \begin{equation}\label{eq:Isomorphism}
        \sqrt{-1}\Li{u}(E)\cong P\times_\rho(\sqrt{-1}\Li{u}(r))
    \end{equation} 
    where $P\times_\rho(\sqrt{-1}\Li{u}(r))$ is the associated vector bundle of $P$ with respect to the representation $\rho$. Define $C_\varepsilon$ to be the subset of $\sqrt{-1}\Li{u}(r)$ consisting of all $\varepsilon$ $2-$positive Hermitian matrices. Then under the isomorphism $\ref{eq:Isomorphism}$,
    \begin{equation}
        C_\varepsilon(E)\cong P\times_\rho C_\varepsilon.
    \end{equation}
    Since $C_\varepsilon$ is closed, so is $C_\varepsilon(E)$. Moreover, $C_\varepsilon(E)$ is invariant under the parallel translation because $C_\varepsilon$ is $U(r)$ invariant. It remains to show that $C_\varepsilon(E)_x$ is convex for each $x\in M$. This is equivalent to the convexity of $C_\varepsilon$. Observe that for any $U\in \sqrt{-1}\Li{u}(r)$ with eigenvalues $\lambda_1(U)\leq\cdots\leq\lambda_r(U)$,
    \[\lambda_1(U)+\lambda_2(U)=\inf_{\substack{\langle v_1,v_2\rangle=0\\|v_1|=|v_2|=1}} \langle Uv_1,v_1\rangle+\langle Uv_2,v_2\rangle.\]
    It follows easily that $\lambda_1+\lambda_2$ is a convex function on $\sqrt{-1}\Li{u}(r)$ and thus $C_\varepsilon$ is convex. By combining the evolution equation \ref{eq:EvolutionContractedCurvature} and the maximum principle \ref{thm:HamiltonMaximumPrinciple}, we deduce (1).

    (2) follows directly from the same argument as in (1) and the last statement in the maximum principle \Cref{thm:HamiltonMaximumPrinciple}.
\end{proof}

\section{Proof of the main theorem}\label{Sec:Proof}

In this section we prove the main theorem \Cref{thm:MainTheorem}. Firstly, we recall the notion of the pseudoindex of Fano manifolds, which is the minimal anti-canonical degree of rational curves:
\begin{defn}[pseudoindex]
    The \emph{pseudoindex} $i(M)$ of a Fano manifold $M$ is 
    \begin{equation}
        i(M)=\inf\{-K_M\cdot C:\text{ $C$ is a rational curve on $M$}\}.
    \end{equation}
\end{defn}

In algebraic geometry, the pseudoindex plays an important role on the classification of Fano manifolds (cf.\cite{CMSB02,Miy04,DH17}). In particular, we shall utilize the following theorem:
\begin{thm}[\cite{Miy04,DH17}]
    If $M^n$ is an $n$ dimensional Fano manifold and $i(M)\ge n$, then $M$ is biholomorphic to either $\C P^{n}$ or the smooth quadric hypersurface $Q^n$.
\end{thm}

By the above structure result, the main theorem follows from an estimate on the aniticanonical degree $-K_M\cdot C$ of any rational curve $C$ on $M$. 

To that end, let $C\subseteq M$ be a rational curve on $M$. By normalization, we get a holomorphic map 
\[f:\C P^1\to C\subseteq M.\]
Let $E=f^*TM$ be the pull-back of the holomorphic tangent bundle of $M$ and $h$ be the pull-back of the K\"{a}hler metric on $TM$. Equip $E$ with the pull back holomorphic structure and let $A_0$ be the Chern connection induced by $h$. Fix the standard Fubini-Study metric $\omega$ on $\C P^1$ and we deform $A_0$ by the Yang-Mills flow \ref{eq:YangMillsFlow}. The convergence theorem \ref{thm:ConvergenceTheorem} shows that $A(t)$ converges to a Yang-Mills connection $A_\infty$. The following lemma is a well known property of Yang-Mills connection on Riemann surfaces (see for example, \cite{AB83}). For reader's convinence, we include a proof here:
\begin{lem}
    $\sqrt{-1}\Lambda_\omega F_{A_\infty}$ is parallel with respect to the connection $A_\infty$.
\end{lem}
\begin{proof}
    Note that the Yang-Mills equation is equivalent to $D_A(\star F_{A_\infty})=0$, where $\star$ is the Hodge star operator on $\End E$ valued 2 forms. Moreover, since the base manifold is a Riemann surface, $\star F_{A_\infty}$ is equal to $\Lambda_\omega F_{A_\infty}$ up to a constant multiple. Hence, $D_A(\sqrt{-1}\Lambda_\omega F_{A_\infty})=0$.
\end{proof}
This lemma implies that $\frac{\sqrt{-1}}{2\pi}\Lambda_\omega F_{A_\infty}$ has constant eigenvalues $a_1\leq a_2\leq\cdots\leq a_n$ everywhere. Suppose that $\lambda_1<\cdots<\lambda_k$ are the distinct eigenvalues. It induces an orthogonal splitting 
\begin{equation}
    E=E_1\oplus\cdots\oplus E_k
\end{equation}
where 
\begin{equation}\label{eq:HermitianYangMills}
    \frac{\sqrt{-1}}{2\pi}\Lambda_\omega F_{A_\infty}|_{E_l}=\lambda_l\Id_{E_l},l=1,\cdots,k.
\end{equation} 
Since $\frac{\sqrt{-1}}{2\pi}\Lambda_\omega F_{A_\infty}$ is parallel, each $E_l$ is invariant under parallel translation. It follows that $E_l$'s are holomorphic subbundles of $(E,D''_{A_\infty})$. Because of \ref{eq:HermitianYangMills}, $h$ is a Hermitian-Einstein metric on $E_l$. By the easy direction of Kobayashi-Hitchin correspondence (cf.\cite{Kob80,Lub83}), $E_l$ is polystable and thus $E_l$ is a direct sum of copies of $\mathcal{O}(\lambda_l)$. Therefore, $a_i$'s are all integers and there is the isomorphism \begin{equation}\label{eq:Splitting}
    E\cong\mathcal{O}(a_1)\oplus\cdots\mathcal\oplus{O}(a_n).
\end{equation}

Next, we show that $a_1+a_2\ge 1$. By the assumption that $M$ has $2-$positive bisectional curvature, $(E,h)$ has $2-$quasi-positive curvature. (Note that we can only get quasi-positivity here because $df$ may vanish at some points of $\C P^1$.) It then follows from (2) of Corollary \ref{cor:MaximumPrincipleYangMills} that $\sqrt{-1}\Lambda_\omega F_{A}$ becomes $2-$positive for some time $t_0>0$. By compactness, we may assume that $\sqrt{-1}\Lambda_\omega F_A$ is $\varepsilon$ $2-$positive at $t_0$ for some positive number $\varepsilon$. Apply (1) of Corollary \ref{cor:MaximumPrincipleYangMills} and we deduce that $\sqrt{-1}\Lambda_\omega F_{A_\infty}$ is also $\varepsilon$ $2-$positive and in particular, 
\begin{equation}\label{eq:FirstInequality}
    a_1+a_2>0\Rightarrow a_1+a_2\ge 1.
\end{equation}

Let $D$ be the divisor consisting of points where $df$ vanishes, counted with order and $\mathcal{O}(D)$ be the associated line bundle. Then $\mathcal{O}(D)\cong\mathcal{O}(k),k\ge 0$. Note that $T_{\C P^1}\otimes\mathcal{O}(E)\cong\mathcal{O}(2+k)$ is a subbundle of $E$. It follows that $a_n\ge 2+k\ge 2$. Since $n\ge 3$, by combining with \ref{eq:FirstInequality}, we get the desired estimate on the anti-canonical degree of the rational curve $C$:
\begin{equation}
    -K_M\cdot C=\langle c_1(E),[\C P^1]\rangle=\sum_{i=1}^na_i=(a_1+a_2)+(a_3+\cdots+a_n)\ge 1+(n-3)+2=n.
\end{equation}
Hence, the pseudoindex $i(M)\ge n$. This completes the proof of the main theorem.

\section{Further Discussions}\label{Sec:FurtherDiscussions}

Continuing under the theme of this note, we talk about some related questions in this section.

The $2-$positivity of the bisectional curvature can be generalized to $m-$positivity on Fano manifolds:
\begin{defn}[$m-$positive bisectional curvature]
    Under the notations of \Cref{Sec:2PositiveBisectionalCurvature}, for any integer $1\leq m\leq n$, $(M,g)$ has $m-$positive bisectional curvature if 
    \[\lambda_1(x,U)+\cdots+\lambda_m(x,U)>0,\quad \forall x\in M, \forall U\in T_xM.\]
\end{defn}
\begin{rk}
    As we noted in the introduction, the $m-$positivity here is different from the one defined in \cite{Siu80}. 
\end{rk}

When $m=n$, the $m-$positivity of the bisectional curvature is just the Ricci positivity and it imposes no further restrictions on Fano manifolds. Therefore, we consider the case when $m<n$. By the same argument as in the proof of the main theorem, one can show that 
\begin{prop}\label{Prop:PseudoindexLowerBound}
    If $n>m$ and $M$ is an $n$ dimensional compact K\"{a}hler manifold with $m-$positive bisectional curvature, then there is a lower bound on the pseudoindex of $M$:
    \begin{equation}
        i(M)\ge n-m+2.
    \end{equation}
\end{prop}  

Given this and Fujita's classification of $n$ dimensional Fano manifolds with pseudoindex $n-1$ (cf.\cite{Fuj90}), our main theorem is expected to be extended to a characterization theorem of compact K\"{a}hler manifolds with K\"{a}hler metrics whose bisectional curvature is $3-$positive. Of course, extra work is required to check whether or not each member in Fujita's list admits K\"{a}hler metrics with $3-$positive bisectional curvature.

It is also interesting to know if the converse of Proposition \ref{Prop:PseudoindexLowerBound} is true or not. When $m=1,2$, it is true by the pseudoindex characterization of $\C P^n$ and $Q^n$, along with our calculation in \Cref{Sec:2PositiveBisectionalCurvature}. The case when $m=3$ might be handled similarly due to Fujita's classification. However, the remaining cases seems unknown. In other words, if $n>m>2$ and $M^n$ is a compact K\"{a}hler manifold with $i(M)\ge n-m+2$, does $M$ admit a K\"{a}hler metric with $m-$positive bisectional curvature?

We are able to prove the main theorem because there are classification theorems when the pseudoindex is large. When the pseudoindex is smaller than $n-1$, there is no longer full classification results available in algebraic geometry. Instead, given the close relation between $m-$positivity and pseudoindex, it is also interesting to consider the differential geometric approach to classifying Fano manifolds based on $m-$positivity with geometric analytic methods such as the K\"{a}hler-Ricci flow. 

There are also various related curvature conditions in literature, for example, those in \cite{NWZ21,Ni21,NZ25}. Here we discuss the positivity of orthogonal Ricci curvature, proposed in \cite{NWZ21}. Let $M^n$ be an $n$ dimensional compact K\"{a}hler manifold. For any holomorphic tangent vector $X$, the orthogonal Ricci curvature along $X$ is defined as 
\begin{equation}
    \Ricci^\perp(X)=\Ricci(X,\Bar{X})-H(X)
\end{equation}
where $H(X)$ is the holomorphic sectional curvature along $X$. $M^n$ has positive orthogonal Ricci curvature if $\Ricci^\perp(X)>0$ for any holomorphic tangent vector $X$. It is clear that the positivity of $\Ricci^\perp$ is weaker than $m-$positivity of bisectional curvature for all $m<n$. Among other things, the paper \cite{NWZ21} deduced a complete classification of 3 dimensional K\"{a}hler manifolds with $\Ricci^\perp>0$:
\begin{thm}[\text{\cite[Theorem 1.8]{NWZ21}}]
    A 3 dimensional compact K\"{a}hler manifold with $\Ricci^\perp>0$ is biholomorphic to either $\C P^3$ or $Q^3$.
\end{thm}
This is a generalization of our main theorem in 3 dimensional case. In dimension 4, the paper also obtains a partial classification result. In higher dimensions, much less is known so far. A better knowledge on $m-$positive bisectional curvature should also be useful towards understanding the positive orthogonal Ricci curvature in higher dimensions.

\bibliographystyle{amsalpha}
\bibliography{references}

\end{document}